\documentclass[a4paper,10pt]{article}
\usepackage[latin1]{inputenc}
\usepackage[english,activeacute]{babel}
\usepackage{array}
\usepackage{amsfonts}
\usepackage{color}
\usepackage{amssymb}
\usepackage{fontenc}
\usepackage{amsthm}
\usepackage{amscd}
\usepackage{tikz-cd}
\usepackage{amsmath}
\usepackage{graphicx}
\usepackage{epsfig}
\usepackage{epstopdf}
\usepackage[margin=3cm]{geometry}
\usepackage{subfigure}
\usepackage{float}

\usepackage[colorlinks=true,linkcolor=blue,draft=false]{hyperref}
\usepackage{aliascnt}

\input xy
\xyoption{all}

\usepackage{algorithm}
\usepackage{algorithmic}

\theoremstyle{plain}
\newtheorem{theorem}{Theorem}[section]

\newaliascnt{lemma}{theorem}
\newtheorem{lemma}[lemma]{Lemma}
\aliascntresetthe{lemma}

\newaliascnt{proposition}{theorem}
\newtheorem{proposition}[proposition]{Proposition}
\aliascntresetthe{proposition}

\newaliascnt{corollary}{theorem}
\newtheorem{corollary}[corollary]{Corollary}
\aliascntresetthe{corollary}

\newaliascnt{conjecture}{theorem}

\aliascntresetthe{conjecture}

\theoremstyle{remark}

\newaliascnt{claim}{theorem}

\aliascntresetthe{claim}

\newtheorem*{claim*}{Claim}
\newtheorem*{remark}{Remark}

\theoremstyle{definition}

\newaliascnt{definition}{theorem}
\newtheorem{definition}[definition]{Definition}
\aliascntresetthe{definition}

\newaliascnt{example}{theorem}

\aliascntresetthe{example}

\newaliascnt{notation}{theorem}
\newtheorem{notation}[notation]{Notation}
\aliascntresetthe{notation}

\begin{document}

\title{On the centralizer of generic braids}
\date{February, 2018}
\author{Juan Gonz\'alez-Meneses and Dolores Valladares\footnote{Both authors partially supported by Spanish Projects MTM2013-44233-P and FEDER. First author partially supported by MTM2016-76453-C2-1-P and FEDER.}}
\maketitle

\begin{abstract}
We study the centralizer of a braid from the point of view of Garside theory, showing that generically a minimal set of generators can be computed very efficiently, as the ultra summit set of a generic braid has a very particular structure. We present an algorithm to compute the centralizer of a braid whose generic-case complexity is quadratic on the length of the input, and which outputs a minimal set of generators in the generic case.
\end{abstract}

\section{Introduction}
Braid groups, introduced by Artin~\cite{Artin,Artin2}, and more generally Artin-Tits groups of spherical type~\cite{Bourbaki,BrieskornSaito}, have a very particular algebraic structure, called Garside structure, which can be used to deduce many porperties, make computations and solve decision problems in these groups. Groups admitting such a Garside structure are called Garside groups~\cite{DehornoyParis,Dehornoy}.

Among the basic results concerning braid groups that are obtained using its Garside structure, one can find the solutions to the word problem \cite{Garside, Epstein, BKL} and the conjugacy problem \cite{Garside, ERM, BKL, FrancoGMconjugacy, Volker, BVGMI, VolkerGMI, VolkerGMII}. Closely related to the latter is the problem of determining the centralizer of a given braid~\cite{Makanin, FrancoGMcentralizer, GMWiest, LeeLee}. In 1971, Makanin~\cite{Makanin} gave an algorithm to compute a generating set of the centralizer $Z(x)$ of any given braid $x$. This method can be easily generalized to all Garside groups, but it has a huge complexity and highly redundant generating sets. A much better algorithm was given in~\cite{FrancoGMcentralizer}, although the actual complexity of the algorithm and the number of generators given as output is not clear.  A different approach uses the geometric properties of braids~\cite{GMWiest}, since the geometric classification of a braid as periodic, reducible or pseudo-Anosov determines the structure of its centralizer. For instance, if $x$ is pseudo-Anosov, one can use the well-known result by McCarthy~\cite{McCarthy, Ivanov} applied to braids~\cite{GMWiest}, to see that $Z(x)$ is isomorphic to $\mathbb{Z}\times\mathbb{Z}$. One generator is pseudo-Anosov, usually a root of $x$, and the other generator is periodic, and can be chosen to be a root of the {\it Garside element} $\Delta^2$. But in order to compute the generators of $Z(x)$, the procedure in~\cite{FrancoGMcentralizer} is usually applied.

The main idea of the algorithm in~\cite{FrancoGMcentralizer} is to compute, given a braid $x$, a finite set of elements called its {\it ultra summit set} ($USS(x)$), which is an invariant subset of the conjugacy class of $x$. The elements of this set are computed along with conjugating elements connecting them, so the ultra summit set is stored as a directed graph, $\Gamma_x$. A generating set of the fundamental group of $\Gamma_x$ yields a generating set of the centralizer of $x$, in a natural way.

The algorithm in~\cite{FrancoGMcentralizer} is very efficient in practice. So the question arises whether, generically, this algorithm will produce a small number of generators in a short time. One can see from computations that, in most cases, the graph $\Gamma_x$ is quite small and satisfies a very particular property: conjugating elements correspond to factors of the normal form of $x$. In this case we will say that $USS(x)$ is {\it minimal} (cf. \autoref{D:minimal_USS}). We will show in this paper that, indeed, this case is generic.

Before stating the result, we need to clarify what do we understand by {\it generic}. We will consider the braid group $B_n$ with some natural generating set (the set of {\it simple elements}). In the Cayley graph of $B_n$ with respect to this set of generators, we will consider the ball $\mathbf B(l)$ of radius $l$ centered at the neutral element.  We will say that a property $\mathcal P$ is {\it generic} in $B_n$, or that the braids satisfying $\mathcal P$ are generic, if the proportion of braids which satisfy $\mathcal P$ in $\mathbf B(l)$ tends to 1 as $l$ tends to infinity.

For instance, Wiest and Caruso~\cite{CarusoBert} showed that braids which admit a `short' conjugation to a {\it rigid} pseudo-Anosov braid are generic (rigidity is a Garside theoretical property that will be explained later). In particular, pseudo-Anosov braids are generic. In~\cite{CarusoBert} it is also proved that braids which admit a given factor in its normal form are generic. Studying with more detail the properties of these generic braids, we show the following:

\noindent\textbf{\autoref{resultadoprincipal}.} {\it The proportion of braids in $\mathbf{B}(l)$ whose ultra summit set is minimal tends to 1 exponentially fast, as $l$ tends to infinity.}

In this generic case in which the ultra summit set of a braid $x$ is minimal, we will be able to describe an explicit minimal set of generators for its centralizer, as follows. In general $USS(x)$ consist of a set of orbits under a special kind of conjugation called {\it cycling}. The conjugating elements for iterated cycling, starting from $x$, will be denoted $a_1, a_2,\ldots$ As this orbit is finite, there is some $k>0$ such that cycling $k$ times $x$ one gets back to $x$. In other words, if we define $PC(x)=a_1\cdots a_k$, then $PC(x)$ commutes with $x$. We will show that if $USS(x)$ is minimal, then it has at most two orbits, and we can describe all possible cases as follows:

\noindent\textbf{\autoref{dosorbitas}.} {\it Let $x$ be a rigid braid such that $USS(x)$ is minimal and has two orbits under cycling. Then $Z(x)=\left\langle PC(x), \Delta^2\right\rangle$.}

\noindent\textbf{\autoref{unaorbita}.} {\it Let $x$ be a rigid braid such that $USS(x)$ is minimal and has only one orbit under cycling of length $k$. One has:
\begin{itemize}
\item [i)] If $\Delta^{-1}x\Delta\neq x$, then $Z(x)=\left\langle a_1\cdots a_{\frac{k}{2}} \Delta^{-1}, \Delta^2\right\rangle$.
\item [ii)] If $\Delta^{-1}x\Delta=x$, then $Z(x)=\left\langle PC(x), \Delta\right\rangle$.
\end{itemize}}

From \autoref{resultadoprincipal} it follows that \autoref{dosorbitas} and \autoref{unaorbita} are describing the centralizers of generic braids. This yields an algorithm to compute a minimal set of generators for the centralizer of a braid, which runs generically in polynomial time.

\noindent \textbf{\autoref{C:polynomial_algorithm}.} {\it There exists an algorithm to compute a generating set for the centralizer of a braid $y\in B_n$, whose generic-case complexity is $\textit{O}(l^2n^4\log n)$, where $l=\ell(y)$.}

The plan of the paper is as follows: in Section 2.1 we describe basic facts about Garside groups. In Section 2.2 and 2.3 we give the basic definitions and results concerning the left normal form and the conjugacy problem, respectively. In Section 2.4 we describe the structure of ultra summit sets, defining the minimal simple elements introduced in~\cite{FrancoGMcentralizer} and two special kind of conjugations from~\cite{BVGMII}. We finish this section with basic notions about the transport map, which will be used several times in this paper. In Section 3, thanks to the work by Caruso and Wiest~\cite{CarusoBert} about genericity of rigid pseudo-Anosov braids, we prove some needed results about generic braids. In Section 4, we define minimal ultra summit sets and we show that the ultra summit set of a generic braid is minimal. Section 5 is devoted to determine the centralizer of generic braids. Finally, in Section 6 we present our algorithm and study its complexity.

\section{Background}
\subsection{Garside groups}

A \textit{Garside group} is a group $G$ which admits a submonoid $P$ and a special element $\Delta\in P$ satisfying some suitable properties~\cite{DehornoyParis}. We will apply Garside properties to the particular case of the braid group $B_n$, so we will just enumerate the needed results in this setting, starting with the classical presentation of $B_n$:
$$
    B_n=\left\langle \sigma_1,\ldots, \sigma_{n-1} \left| \begin{array}{cl}
      \sigma_i\sigma_j=\sigma_j\sigma_i & \mbox{ if $|j-i|>1$} \\
      \sigma_i\sigma_j\sigma_i=\sigma_j\sigma_i\sigma_j & \mbox{ if $|j-i|=1$}
    \end{array}\right.\right\rangle.
$$
The elements of $B_n$ which can be written as a product positive powers of generators are called {\it positive elements}, and they form the monoid $P$, usually denoted $B_n^+$. The special element $\Delta$, called {\it Garside elment}, is:
$$
   \Delta=(\sigma_1\sigma_2\cdots\sigma_{n-1}) (\sigma_1\sigma_2\cdots\sigma_{n-2})\cdots (\sigma_1\sigma_2) \sigma_1.
$$
Given two braids $a,b\in B_n$, we say that $a$ is a prefix of $b$, and we write $a\preccurlyeq b$, if $a^{-1}b\in B_n^+$. This determines a partial order in $B_n$ which is invariant under left-multiplication, and is a lattice order. We denote $a\wedge b$ (resp. $a\vee b$) the meet (resp. the join) of the braids $a$ and $b$ with respect to $\preccurlyeq$. We will also say that $a\wedge b$ is the greatest common prefix of $a$ and $b$, and that $a\vee b$ is their least common multiple.

The set $[1,\Delta]=\{s \in B_n| 1\preccurlyeq s \preccurlyeq \Delta\}\subset B_n^+$, called the set of \textit{simple elements} of $B_n$, is a very particular finite set of generators of $B_n$, which are the building blocks of the normal forms that we will see later.

The generators $\sigma_1,\ldots, \sigma_n$ are called {\it atoms}. Conjugation by the Garside element $\Delta$ determines an inner automorphism of $B_n$ that we denote $\tau$ (that is, $\tau(x)=\Delta^{-1}x\Delta$). It is easy to check that $\tau(\sigma_i)=\sigma_{n-i}$ for every $i=1,\ldots, n-1$, so $\tau$ has order 2 and $\Delta^2$ is central. Actually, the center of $B_n$ is cyclic, generated by $\Delta^2$.

Throughout this paper we will focus on the braid group $B_n$, although some results could be extended or adapted to other Garside groups, as their proofs basically involve techniques from Garside theory.

\subsection{Left normal form}

Let us see how the Garside structure of $B_n$ allows to define a particular unique decomposition of each braid, called its {\it left normal form}.

\begin{definition} \rm Given a simple element $s\in [1,\Delta]$, the \textit{right complement} of $s$ is defined by $\partial(s)=s^{-1}\Delta$, and the \textit{left complement} of $s$ is
$\partial^{-1}(s)=\Delta s^{-1}$.
\end{definition}

Notice that the map $\partial$ is a bijection of $[1,\Delta]$ and that $\partial^2=\tau$ \cite{BVGMI}.

\begin{definition} \rm\cite{Epstein}\label{leftweighted} Let $s$, $t$ be two simple elements in $B_n$. We say that the decomposition $st$ is \textit{left-weighted} if $\partial(s)\wedge t=1$ or, equivalently, $\Delta\wedge st=s$.
\end{definition}

\begin{definition} \rm\cite{Epstein} Given $x\in B_n$, we say that a decomposition $x=\Delta^p x_1\ldots x_l$, where $p\in\mathbb{Z}$ and $l\geq 0$, is the \textit{left normal form} of $x$ if $x_i \in [1, \Delta]\setminus \{1, \Delta\}$ for $i=1,\ldots, l$ and $x_ix_{i+1}$ is a left-weighted decomposition for $i=1,\ldots, l-1$.
\end{definition}

\begin{definition} \rm\cite{Epstein} Given $x \in B_n$ whose left normal form is $\Delta^p x_1\ldots x_l$, we define the \textit{infimum}, \textit{supremum} and \textit{canonical length} of $x$, respectively, by $\inf(x)=p$, $\sup(x)=p+l$ and $\ell(x)=l$.
\end{definition}

It is well known that left normal forms of elements in $B_n$ exist and are unique~\cite{Epstein,ERM}. The right complement plays an important role when comparing the left normal forms of $x$ and $x^{-1}$.

\begin{theorem} \rm{\cite{BVGMI}}\label{normalforminverso} If $x=\Delta^p x_1\ldots x_l$ is in left normal form, then the left normal form of $x^{-1}$ is equal to $x^{-1}=\Delta^{-p-l} x^{\prime}_l\ldots x^{\prime}_1,$, where $x^{\prime}_i=\tau^{-p-i}(\partial(x_i))$, for $i=1,\ldots,l$.
\end{theorem}

As we saw in \autoref{leftweighted}, the decomposition $st$ is left-weighted if $\partial(s)$ and $t$ have no prefixes in common, except the trivial one. As $\partial(s)$ is the simple element such that $s\partial(s)=\Delta$, the prefixes of $\partial(s)$ are precisely those simple elements $s^\prime$ for which $ss^\prime$ is simple. Hence, the decomposition $st$ is left-weighted if and only if the only prefix $s^\prime$ of $t$ such that $ss^\prime$ is simple is the trivial one.

Given two simple elements $s$ and $t$, it is easy to find the left-weighted decomposition of the product $st$ using the description above. If the decomposition $st$ is not left-weighted, this means that there is a nontrivial prefix $s^\prime\preccurlyeq t$  such that $ss^\prime$ is simple. Since $\preccurlyeq$ is a lattice order, there is a maximal element satisfying this property, namely $s^\prime=\partial(s)\wedge t$. Hence, transforming the decomposition $st$ into a left-weighted one means to slide the prefix $s^\prime$ from the second factor to the first one. That is, write $t=s^\prime t^\prime$ and then consider the decomposition $st=$($ss^\prime$)$t^\prime$, with $ss^\prime$ as the first factor and $t^\prime$ as the second one. Since $s^\prime$ is maximal, the decomposition ($ss^\prime$)$t^\prime$ is left-weighted.

The action of obtaining the left-weighted decomposition of the product of two simple elements $s$ and $t$, by sliding the simple element $s^\prime$ from the second factor to the first factor, is known as a \textit{local sliding} applied to the decomposition $st$. Using local slidings one can compute the left normal form of any element of a Garside group.

\subsection{The conjugacy problem}

The known algorithms to solve the conjugacy problem in braid groups share a common strategy. Given an element $x\in B_n$, the idea is to compute a finite subset of the conjugacy class $x^{B_n}$ of $x$, which consists of those conjugates satisfying some suitable conditions, and which only depends on $x^{B_n}$, not on $x$ itself. There are two special types of conjugations ever since the work of ElRifai and Morton in \rm\cite{ERM}:

\begin{definition} \rm Given $x=\Delta^p x_1\cdots x_l$ written in left normal form, where $l>0$,

\noindent the \textit{cycling} of $x$ is: $$\textbf{c}(x)=\Delta^p x_2\cdots x_l\tau^{-p}(x_1),$$
and the \textit{decycling} of $x$ is: $$\textbf{d}(x)=x_l \Delta^p x_1\cdots x_{l-1}.$$
\end{definition}

The conjugating elements involved in a cycling or a decycling will play a crucial role later, so to avoid repeated references to the automorphism $\tau$ we introduce the following definition.

\begin{definition} \rm{\cite{BVGMI}} Let $x$ be a braid with left normal form $x=\Delta^p x_1\cdots x_l$ ($l>0$). We define the \textit{initial factor} of $x$ as $\iota(x)=\tau^{-p}(x_1)$, and the \textit{final factor} of $x$ as $\varphi(x)=x_l$. If $l=0$ we define $\iota(\Delta^p)=1$ and $\varphi(\Delta^p)=\Delta$.
\end{definition}

Notice that, up to conjugation by $\Delta^p$, the simple element $\iota(x)$ (resp. $\varphi(x)$) corres\-ponds to the first (resp. last) non-$\Delta$ factor in the left normal form of $x$. From \autoref{normalforminverso}, the initial and final factors of $x$ and $x^{-1}$ are closely related.

\begin{lemma} \rm{\cite{BVGMI}}\label{relacionfactoresxeinverso} For every $x$ $\in B_n$ one has $\iota(x^{-1})=\partial(\varphi(x))$ and $\varphi(x^{-1})=\partial^{-1}(\iota(x))$.
\end{lemma}

\begin{lemma} \rm{\cite{BVGMII}} Given $x\in B_n$ with $\ell(x)>0$, one has:
$$ \begin{array}{ccc}
     x^{\iota(x)}=\textbf{c}(x), & x^{{\varphi(x)}^{-1}}=\textbf{d}(x), &  x^{\iota(x^{-1})}=x^{\partial(\varphi(x))}=\tau(\textbf{d}(x)).
   \end{array}$$
\end{lemma}

We define the \textit{twisted decycling} of $x$ as $\tau(\textbf{d}(x))$.

Going back to the conjugacy problem, the finite invariant subset of $x^{B_n}$ defined in~\cite{ERM} is called the \textit{super summit set} of $x$, denoted $SSS(x)$. It consists of the conjugates of $x$ having minimal canonical length. This minimal canonical length among the conjugates of $x$ is called the {\it summit length} of $x$, and denoted $\ell_s(x)$. In~\cite{ERM} it was shown that one obtains an element in $SSS(x)$, starting with $x$, by iterated application of cyclings and decyclings.

We will not consider $SSS(x)$ in this paper, but the following invariant subset of $SSS(x)$ introduced by Gebhardt:

\begin{definition} \rm\cite{Volker} The \textit{ultra summit set} $USS(x)$ of $x\in B_n$ is the set of elements $y\in SSS(x)$ such that $\textbf{c}^m(y)=y$, for some $m>0$.
\end{definition}

We remark that one obtains an element in $USS(x)$ by iterated application of cycling to an element in $SSS(x)$. Thus, the ultra summit set $USS(x)$ consists of a finite set of disjoint, closed orbits under cycling. Once one obtains an element in $USS(x)$, it is explained in~\cite{Volker} how to compute all elements in $USS(x)$, together with conjugating elements connecting them.

This was improved in \rm\cite{VolkerGMI} by introducing a new kind of conjugation, called cyclic sliding, to replace cycling and decycling. Cyclic sliding simplifies the algorithms concerning conjugacy in Garside groups, and in particular in braid groups.

\begin{definition} \rm\cite{VolkerGMI} Given $x \in B_n$, its \textit{preferred prefix} is $\mathfrak{p}(x)=\iota(x)\wedge\iota(x^{-1})=\iota(x) \wedge \partial(\varphi(x))$.
\end{definition}

\begin{definition} \rm\cite{VolkerGMI} For $x \in B_n$, the \textit{cyclic sliding} of $x$ is $\mathfrak{s}(x)=x^{\mathfrak{p}(x)}=\mathfrak{p}(x)^{-1} \ x \ \mathfrak{p}(x)$.
\end{definition}

Using cyclic sliding one can define a new subset of the conjugacy class of a braid $x$, which is contained in $USS(x)$. If one applies iterated cyclic sliding to an element $x\in B_n$, one eventually obtains a repeated element, say $y$. The orbit of $y$ under cyclic sliding will be called a sliding circuit. The set of all (disjoint) sliding circuits in the conjugacy class of $x$ is the mentioned invariant subset:

\begin{definition} \rm\cite{VolkerGMI} The \textit{set of sliding circuits} $SC(x)$ of $x\in B_n$ is the set of conjugates $y\in x^{B_n}$ such that $\mathfrak{s}^m(y)=y$, for some $m>0$.
\end{definition}

In this paper we will study the centralizer of a special kind of braids, called rigid braids. The idea of studying rigidity came from the study of elements whose left normal form changes only in the obvious way under cyclings, decyclings and powers, so their ultra summit sets are easier to study. Rigid elements in Garside groups were studied in {\rm\cite{BVGMI}}.

\begin{definition} \rm\cite{BVGMI} Let $x=\Delta^p x_1\ldots x_l$ be in left normal form, with $l>0$. Then $x$ is \textit{rigid} if $\varphi(x) \iota(x)$ is left-weighted as written, that is, if $\mathfrak{p}(x)=1$.
\end{definition}


\subsection{Minimal simple elements}

Let us focus on the ultra summit set $USS(x)$ of a braid $x\in B_n$. As we said earlier, the elements of $USS(x)$ are computed along with conjugating elements connecting them. We will describe in detail those conjugating elements. We shall use the notation $y^s=s^{-1}ys$.

\begin{definition} \rm\cite{Volker} Given $x\in B_n$ and $y\in USS(x)$, we say that a simple element $1\neq s\in [1,\Delta]$ is a \textit{minimal simple element} for $y$ (with respect to $USS(x)$) if $y^s \in USS(x)$, and $y^t\notin USS(x)$ for every $1\precneqq t \precneqq s$.
\end{definition}

\begin{definition}\label{grafoUSS} \rm Given $x\in B_n$, the directed graph $\Gamma_x$ is defined as follows:
\begin{itemize}
\item The vertices are the elements of $USS(x)$.
\item There is an arrow with label $s$ going from $y$ to $y^s$, for every minimal simple element $s$ for $y$.
\end{itemize}
\end{definition}

In \rm\cite{Volker} Gebhardt described how to compute the minimal simple elements for a given $y\in USS(x)$. Moreover, he shows the following:

\begin{theorem} {\rm(\cite[Theorem 1.18 and Corollary 1.19]{Volker})}\label{existenciaminimalsimpleelements} \rm Let $x\in B_n$ and $y\in USS(x)$.
\begin{itemize}
\item [i)] If $s,t \in B_n$ are such that $y^s\in USS(x)$ and $y^t\in USS(x)$, then $y^{s\wedge t}\in USS(x)$.
\item [ii)] For every $u\in B_n^+$ there is a unique element $c_y(u)$ which is minimal with respect to $\preccurlyeq$ among the elements $\upsilon$ satisfying $u\preccurlyeq \upsilon$ and $y^{\upsilon}\in USS(x)$.
\item [iii)] The graph $\Gamma_x$ described in \autoref{grafoUSS} is finite and connected. Its transitive closure is a complete graph, i.e., every vertex is reachable from every other vertex by an oriented path.
\end{itemize}
\end{theorem}

By the above result, given one element $\widetilde x\in USS(x)$, one can obtain any other element in $USS(x)$ just through conjugations by minimal simple elements. In this way one can compute the whole graph $\Gamma_x$, starting with a single element $\widetilde x\in USS(x)$.

From now on, to simplify the notation, we will assume that $x$ belongs to its own ultra summit set. That is, $x\in USS(x)$. It is known that conjugations by minimal simple elements are quite special:

\begin{corollary} \rm\cite[Theorem 2.5]{BVGMII} \label{caracterizacionelementosminimales} Let $x\in USS(x)$ with $\ell(x)>0$ and let $s$ be a minimal simple element for $x$. Then $s$ is a prefix of either $\iota(x)$ or $\iota(x^{-1})$, or both.
\end{corollary}

\begin{definition} \rm Let $x\in B_n$. A \textit{partial cycling} of $x$ is a conjugation of $x$ by a prefix of $\iota(x)$. A \textit{partial twisted decycling} of $x$ is a conjugation of $x$ by a prefix of $\iota(x^{-1})=\partial(\varphi(x))$.
\end{definition}

Consequently, by \autoref{existenciaminimalsimpleelements} and \autoref{caracterizacionelementosminimales}, given $x,y\in USS(x)$, there exists a sequence of partial cyclings and partial twisted decyclings joining $x$ to $y$.

Due to \autoref{caracterizacionelementosminimales}, there are two kinds of minimal simple elements, hence there are two kinds of arrows in $\Gamma_x$. Following \cite{BVGMII}, we say that an arrow $s$ starting at a vertex $y\in USS(x)$ is \textit{black} if $s$ is a prefix of $\iota(y)$, and it is \textit{grey} if $s$ is a prefix of $\iota(y^{-1})$. In other words, an arrow starting at $y$ is black if it corresponds to a partial cycling of $y$, and it is grey if it corresponds to a partial twisted decycling of $y$. Notice that an arrow can, a priori, be black and grey at the same time. In that case, we say it is a bi-colored arrow.

\begin{definition} \rm A \textit{path} in $\Gamma_x$ is a (possibly empty) sequence $(s_1^{e_1},\ldots, s_k^{e_k})$, where $s_i$ is an arrow in $\Gamma_x$ and $e_i=\pm 1$, such that the endpoint of $s_i^{e_i}$ is equal to the starting point of $s_{i+1}^{e_{i+1}}$ for every $i=1,\ldots,k-1$. We say that a path $(s_1^{e_1},\ldots, s_k^{e_k})$ is \textit{oriented} if $e_i=1$ for $i=1,\ldots,k$.
\end{definition}

\begin{remark} \rm Every path $(s_1^{e_1},\ldots, s_k^{e_k})$ determines an element $\alpha = s_1^{e_1}\ldots s_k^{e_k}$.  Distinct paths may determine the same element. Since the labels of arrows are simple elements, it follows that if the path is oriented then $\alpha \in B_n^+$.
\end{remark}

We end this subsection with the case of rigid braids, since the structure of $\Gamma_x$ we described above will be simpler in this case. Concerning the arrows of $\Gamma_x$, the main difference between the case of a rigid braid and the general case is the following:

\begin{lemma} \rm{\cite{BVGMII}} Let $x\in USS(x)$ and $\ell(x)>0$. Then $x$ is rigid if and only if there are no bi-colored arrows starting at $x$.
\end{lemma}

Refer to {\rm\cite[\S 2.1]{BVGMII}} to see several examples illustrating the notions we have seen in this subsection.

\subsection{The transport map}

We finish this introductory section explaining a tool that will be used several times in this paper: The transport map. Given two conjugate braids $x$ and $x^\alpha=\alpha^{-1}x\alpha$, the images of $x$ and $x^\alpha$ under cycling are also conjugate, and we can relate $\alpha$ to a conjugating element for the images $\textbf{c}(x)$ and $\textbf{c}(x^\alpha)$.

\begin{definition} \rm\cite{Volker} Given $x,\alpha \ \in$ $B_n$, we define the \textit{transport} of $\alpha$ at $x$ under cycling as
$$\alpha^{(1)}=\iota(x)^{-1} \alpha \iota(x^\alpha).$$
That is, $\alpha^{(1)}$ is the conjugating element that makes the following diagram commutative, in the sense that the conjugating element along any closed path is trivial:

$$\begin{CD}
x @>\iota(x)>> \textbf{c}(x)\\
@V\alpha VV @VV\alpha^{(1)}V\\
x^{\alpha} @>\iota(x^{\alpha})>> \textbf{c}(x^{\alpha})
\end{CD}$$

Note that the horizontal rows in this diagram correspond to applications of cycling.
For an integer $i>1$ we define recursively $\alpha^{(i)}=(\alpha^{(i-1)})^{(1)}$, which is the transport of $\alpha^{(i-1)}$ at $\textbf{c}^{i-1}(x)$. We also define $\alpha^{(0)}=\alpha$.
\end{definition}

Under certain conditions, the transport under cycling respects many aspects of the Garside structure. In particular, if $x$ and $x^\alpha$ as above are super summit elements, that is, $x$ and $x^\alpha$ have minimal canonical length in their conjugacy class, the transport map preserves products, left divisibility, greatest common prefixes, and powers of $\Delta$. For more details, refer to {\rm\cite{Volker}}.

Suppose that a given braid $x$ is rigid. Then $x$ belongs to its ultra summit set (it trivially belongs to $SC(x)$, which is contained in $USS(x)$),  and the minimal simple elements for $x$ have some particular properties that we will show to end this section.  These properties can be deduced from the results in \cite{BVGMII}, but we will provide proofs.

\begin{lemma}\label{transportfactoressimplesNF} Let $x$ be a rigid braid with normal form $\Delta^p x_1\cdots x_l$ $(l>0)$, and let $a=\iota(x)$. Then $a^{(kl+r)}=\tau^{-(k+1)p}(x_{r+1})$, for all $k\geq 0$ and all $r=0,\ldots,l-1$.
\end{lemma}

\begin{proof}
For every $i>0$, we have the following diagram:
$$
\begin{CD}
x @>\iota(x)>>\textbf{c}(x) @>\iota(\textbf{c}(x))>> \textbf{c}^2(x) @>>> \cdots @>>> \textbf{c}^{i-1}(x) @>\iota(\textbf{c}^{i-1}(x))>> \textbf{c}^i(x)\\
@VVa=\iota(x)V @VVa^{(1)}V @VVa^{(2)}V @. @VVa^{(i-1)}V @VVa^{(i)}V\\
\textbf{c}(x) @>\iota(\textbf{c}(x))>> \textbf{c}^2(x) @>\iota(\textbf{c}^2(x))>> \textbf{c}^3(x) @>>> \cdots @>>> \textbf{c}^{i}(x) @>{\rlap{$\scriptstyle{\ \ \ \iota(\textbf{c}^i(x))}$}\phantom{\text{very long label}}}>> \textbf{c}^{i+1}(x)
\end{CD}
$$

By definition of the transport of $a$ at $x$,
\begin{eqnarray*}a^{(i)} & = &\iota(\textbf{c}^{i-1}(x))^{-1}\cdots \iota(\textbf{c}(x))^{-1} \; \iota(x)^{-1}\; \iota(x)\; \iota(\textbf{c}(x))\; \cdots\iota(\textbf{c}^{i-1}(x))\; \iota(\textbf{c}^{i}(x)) \\ & = &\iota(\textbf{c}^{i}(x))
\end{eqnarray*}

By hypothesis, $a=\iota(x)=\tau^{-p}(x_1)$, so the result is true for $k,r=0$. Since $x$ is rigid, $\textbf{c}(x)=\Delta^p x_2\cdots x_l a$ is in left normal form as written. Then $\textbf{c}(x)$ is also rigid. Applying the same argument to each cycling of $x$, it follows that $\textbf{c}^i(x)$ is rigid for every $i>0$, and that cycling just corresponds to cyclic permutation of the factors of $x$ (up to conjugation by some power of $\Delta$). More precisely, every $l$ cyclings, the factors of the normal form of $x$ go back to their original position, but each one is conjugated by $\Delta^{-p}$. Therefore $a^{(kp+r)}=\iota(\textbf{c}^{kp+r}(x))=\tau^{-(k+1)p}(x_{r+1})$, as we wanted to show.
\end{proof}

We will use several times the following important property, which shows that, in $USS(x)$, iterated transport of an element always comes back to the original element.

\begin{lemma}\label{L:transport_fixed_element}{\rm \cite[Lemma 2.6]{Volker}}
Let $x$ belong to Garside group, and let $u$ be a positive element. Assume that $x, x^u\in USS(x)$. Let $m$ be a positive integer such that $\mathbf c^m(x)=x$ and $\mathbf c^{m}(x^u)=x^u$. Then $u^{(km)}=u$ for some $k>0$.
\end{lemma}

We can deduce the following property for minimal simple elements.

\begin{lemma}\label{transportminimal} Let $x\in USS(x)$ and let $u$ be a minimal simple element for $x$ with respect to $USS(x)$. Then, $u^{(i)}$ is a minimal simple element for $\mathbf c^{i}(x)$, for all $i\geq 1$.
\end{lemma}

\begin{proof}
Suppose that $u^{(i)}$ is not a minimal simple element for some $i$. So there exist nontrivial positive elements $a, b$ such that $u^{(i)}=a b$ and $(\mathbf c^i(x))^{a}\in USS(x)$. By~\autoref{L:transport_fixed_element}, iterated transport of $a$ comes back to $a$, hence $a^{(j)}\neq 1$ for every $j>0$ (as if a trivial element is obtained, all the forthcoming transports would be trivial). In the same way, $b^{(j)}\neq 1$ for every $j>0$. But we know, by~\autoref{L:transport_fixed_element} again, that $u^{(i+t)}=u$ for some $t>0$. As transport preserves products, we have $u=u^{(i+t)}=(u^{(i)})^{(t)}=(ab)^{(t)}=a^{(t)}b^{(t)}$, where $a^{(t)}\neq 1$, $b^{(t)}\neq 1$ and, by construction, $x^{(a^{(t)})}\in USS(x)$. This contradicts the minimality of $u$.
\end{proof}

\begin{definition} Let $x$ be a rigid braid (so $x\in USS(x)$) with left normal form $\Delta^p x_1\cdots x_l$. We say that a simple factor of the normal form of $x$, say $x_r$, is \textit{minimal} if $\tau^{-p}(x_r)$ is a minimal simple element for $\mathbf c^{r-1}(x)$ with respect to $USS(x)$.
\end{definition}

As a direct consequence of \autoref{transportminimal} and \autoref{transportfactoressimplesNF}, we have the following.

\begin{corollary}\label{simplefactorNFsminimal} Let $x$ be a rigid braid whose left normal form is $\Delta^p x_1\cdots x_l$. If a simple factor of its normal form is minimal, then all factors are minimal.
\end{corollary}

\begin{proof}
Suppose $x_r$ is minimal for some $r=1, \ldots, l$. This means that $\tau^{-p}(x_r)$ is a minimal simple element for $\textbf{c}^{r-1}(x)$ with respect to $USS(x)$. Since $x$ is rigid, $\textbf{c}^{r-1}(x)$ is also rigid. Hence, by \autoref{transportminimal}, $(\tau^{-p}(x_r))^{(i)}$ is a minimal simple element for every $i>0$.

On the other hand, by \autoref{transportfactoressimplesNF}, every simple factor of the normal form of $x$ will eventually appear as a transport of $\tau^{-p}(x_r)$, up to conjugation by some power of $\Delta$. Since conjugation by $\Delta$ preserves minimal simple elements, the result follows.
\end{proof}

\section{Generic braids}

We will now present some results by Caruso and Wiest~\cite{Caruso,CarusoBert} on generic braids, and adapt them to our purposes.

Let $\mathbf B(l)$ be the ball of radius $l$ centered at 1 in the Cayley graph of the braid group $B_n$, with generators the simple braids. We are interested in the proportion of braids in $\mathbf B(l)$ which have a very particular ultra summit set, when $l$ tends to infinity. In order to find this proportion, we will follow the arguments in~\cite{CarusoBert}.

Let $ B_n^{p,l}=\{x \in B_n \ | \ inf(x)=p, \ell(x)=l\}$. The sets $B_n^{p,l}$ are disjoint as left normal forms are unique. Since left normal forms are closely related to the so-called mixed normal forms~\cite{CharneyMeier}, and the latter are geodesics in the mentioned Cayley graph of $B_n$, is follows that
\begin{equation}\label{B(l)_decomposed}
    \mathbf B(l)=\bigsqcup_{k=0}^{l} \bigsqcup_{\eta=-l}^{l-k}B_n^{\eta,k}.
\end{equation}
So $\mathbf B(l)$ is the disjoint union of a finite number of sets of the form $B_n^{\eta,k}$. Also, it is shown in~\cite{Caruso} that $\left|B_n^{\eta,k}\right|=\Theta(\lambda^k)$ for some $\lambda>1$, meaning that the sequences $\frac{\left|B_n^{\eta,k}\right|}{\lambda^k}$  and $\frac{\lambda^k}{\left|B_n^{\eta,k}\right|}$ stay bounded as $k$ tends to infinity.

\begin{definition} \rm\cite{CarusoBert} Let $x$ be a braid with left normal form $x=\Delta^{p} x_1\cdots x_l$. Let us denote $P(x)=x_{2\cdot \lceil\frac{l}{5}\rceil +1} \cdots x_{l- 2\cdot \lceil\frac{l}{5}\rceil}$ (the middle fifth of the left normal form). A conjugation of $x$ is \textit{non-intrusive} if the normal form of the conjugated braid contains $P(x)$ as a subword.
\end{definition}

\begin{proposition}\label{rigid_pA_are generic} {\rm\cite{CarusoBert}} There exists a constant $0<\mu_R<1$ (which depends only on $n$) such that, among the braids in $B_n^{p,l}$, the proportion of those that can be non-intrusively conjugated to a rigid pseudo-Anosov braid is at least $1-\mu^l_R$, for sufficiently large $l$.
\end{proposition}

\begin{lemma}\label{dividirNF} {\rm\cite{Caruso}} Let $w$ be a fixed braid of infimum 0. Then there exists a constant $0<\mu_{M}<1$ such that the proportion of braids $x\in B_n^{p,l}$ for which $w$ appears as a set of consecutive factors of the normal form of $P(x)$ is at least $1-\mu^l_{M}$, for sufficiently large $l$.
\end{lemma}

We are interested in the following kind of braids:

\begin{definition} Let $x$ be a braid with left normal form $x=\Delta^p x_1\cdots x_l$. We say that $x$ is a \textit{$\sigma_1$-non-intrusive braid} if $x$ admits a non-intrusive conjugation to a rigid pseudo-Anosov braid and $\sigma_1$ and $\partial(\sigma_1)$ are simple factors of $P(x)$.
\end{definition}

\begin{corollary}
There exists a positive constant $0<\mu<1$ such that the proportion of $\sigma_1$-non-intrusive braids in $B_n^{p,l}$ is at least $1-\mu^l$, for sufficiently large $l$.
\end{corollary}

\begin{proof}
We just need to take $\mu>\max(\mu_R,\mu_{M_1}, \mu_{M_2})$, where $\mu_R$ is the constant appearing in~\autoref{rigid_pA_are generic}, $\mu_{M_1}$ is the constant appearing in~\autoref{dividirNF} for $w=\sigma_1$, and $\mu_{M_1}$ is the constant appearing in~\autoref{dividirNF} for $w=\partial(\sigma_1)$.
\end{proof}

\begin{corollary}\label{genericitynonintrusivebraid} The proportion of $\sigma_1$-non-intrusive braids in $\mathbf{B}(l)$ tends to 1 as $l$ tends to infinity. Moreover, this convergence happens exponentially fast.
\end{corollary}

\begin{proof}
This is similar to the proof of  Theorem 5.1 in~\cite{CarusoBert}. The ball $\mathbf B(l)$ can be decomposed as in~(\ref{B(l)_decomposed}), where the size of each $B_n^{\eta,k}$ is $\Theta(\lambda^k)$ for some $\lambda>1$. In each $B_n^{\eta,k}$, the proportion of $\sigma_1$-non-intrusive braids is at least $1-\mu^k$ for sufficiently large $k$. Let $k_0$ be the biggest value of $k$ such that the mentioned proportion is smaller than $1-\mu^k$. Then, for $l>k_0$, we can decompose
$$
   \mathbf B(l)=\left(\bigsqcup_{k=0}^{k_0} \bigsqcup_{\eta=-l}^{l-k}B_n^{\eta,k}\right)\bigsqcup\left(\bigsqcup_{k=k_0+1}^{l} \bigsqcup_{\eta=-l}^{l-k}B_n^{\eta,k}\right).
$$
The size of the set on the left is $O((k_0+1)(2l+1)\lambda^{k_0})$, that is, $O(2l+1)$. In the set on the right, the number of braids which are {\bf not} $\sigma_1$-non-intrusive is
$$
O\left((2l-k_0)(\lambda\mu)^{k_0+1}+(2l-k_0-1)(\lambda\mu)^{k_0+2}+\cdots (l+2)(\lambda\mu)^{l-1}+(l+1)(\lambda\mu)^{l}\right).
$$
Now notice that the number of elements in $\mathbf B(l)$ is at least $|B_n^{0,l}|=\Theta(\lambda^l)$. Therefore, the proportion of  braids which are {\bf not} $\sigma_1$-non-intrusive in $\mathbf B(l)$ is
$$
   O\left( \frac{2l+1}{\lambda^l}+\frac{(2l-k_0)\mu^{k_0+1}}{\lambda^{l-k_0-1}}+ \frac{(2l-k_0-1)\mu^{k_0+2}}{\lambda^{l-k_0-2}} +\cdots + \frac{(l+2)\mu^{l-1}}{\lambda} + (l+1)\mu^l \right)
$$
$$
   \leq O\left((2l+1)(l+1)\left(\max(\lambda^{-1},\mu)\right)^l\right).
$$
This tends exponentially fast to 0, hence the proportion of $\sigma_1$-non-intrusive braids in $\mathbf B(l)$ tends exponentially fast to 1, as $l$ tends to infinity.
\end{proof}

\section{Minimal ultra summit set}

Once we showed that $\sigma_1$-non-intrusive braids are generic, we will see that their ultra summit sets are particularly simple. Recall the directed graph $\Gamma_x$ related to the ultra summit set $USS(x)$ of a braid $x$, and recall also that the arrows in $\Gamma_x$ can be black or grey (or both). Since $USS(x)\subset SSS(x)$, the canonical length of all elements in $USS(x)$ is the same, namely $\ell_s(x)$.

\begin{definition}\label{D:minimal_USS} Let $x\in B_n$. We say that $USS(x)$ is {\it minimal} if $\ell_s(x)>1$ and, for every vertex $y$ in $\Gamma_x$ there is only one black arrow starting at $y$, corresponding to $\iota(y)$, and only one grey arrow starting at $y$, corresponding to $\iota(y^{-1})=\partial(\varphi(y))$.
\end{definition}

\begin{lemma}\label{minimalrigidity} Let $x\in B_n$. If $USS(x)$ is minimal, then all elements in $USS(x)$ are rigid.
\end{lemma}

\begin{proof}
Recall that $y$ is rigid if and only if $\mathfrak p(y)=\iota(y)\wedge \iota(y^{-1})=1$. Suppose that $USS(x)$ is minimal, that is, for every $y\in USS(x)$, $\iota(y)$ and $\iota(y^{-1})$ are minimal simple elements for $y$. This means that $y^{\iota(y)}$ and $y^{\iota(y^{-1})}$ belong to $USS(x)$ and no proper positive prefix of either conjugating element conjugates $y$ to $USS(x)$. Let $s=\iota(y)\wedge \iota(y^{-1})$. By~\autoref{existenciaminimalsimpleelements}, $y^s \in USS(x)$. Since $s$ is a prefix of both $\iota(y)$ and $\iota(y^{-1})$, which are minimal, there are just two possible cases: Either $s=\iota(y)=\iota(y^{-1})$ or $s=\iota(y)\wedge \iota(y^{-1})=1$.

 Suppose that $s=\iota(y)=\iota(y^{-1})$. That is, $\iota(y)=\partial(\varphi(y))$. We would then have $\varphi(y)\iota(y) = \varphi(y)\partial(\varphi(y)) = \Delta$. But then, if $\Delta^p y_1\cdots y_l$ is the left normal form of $y$, we would have $\mathbf c(y)=\Delta^p y_2\cdots y_{l-1} \varphi(y) \iota(y) = \Delta^p y_2\cdots y_{l-1} \Delta $, hence the canonical length of \textbf{c}($y$) would be smaller than the canonical length of $y$, contradicting the fact that $y\in USS(x)$. Hence, $\iota(y)\wedge \iota(y^{-1})=1$, meaning that $y$ is rigid for every $y \in USS(x)$.
\end{proof}

Suppose that $USS(x)$ is minimal, so by~\autoref{minimalrigidity} it consists of rigid elements. By definition, at every vertex $y$ of $\Gamma_x$ there are two outgoing arrows, a black one corresponding to $\iota(y)$, and a grey one corresponding to $\iota(y^{-1})=\partial(\varphi(y))$. But once the conjugation by $\iota(y)$ is performed, as $y$ is rigid, the last factor of the obtained element is precisely $\iota(y)$. This means that every vertex of $\Gamma_x$ has exactly one incoming black arrow, which corresponds to its final factor. In the same way, as $y^{-1}$ is rigid (rigidity is preserved by taking inverses), the same argument applies to the grey arrows: Every vertex of $\Gamma_x$ has exactly one incoming grey arrow, which corresponds to the final factor of its inverse.

Therefore, the graph $\Gamma_x$, locally at $y$, is as sketched in~\autoref{figproductoflechas}, where the grey arrows are represented by dashed lines.

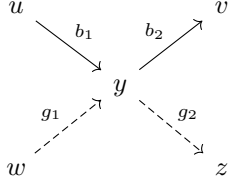
\begin{figure}[H]
\centering
    \begin{tikzcd}%
        u \arrow[rd, "b_1"] &  & v  \\
          &  y \arrow[ru, "b_2"] \arrow[rd, dashrightarrow, "g_2"] & \\
        w \arrow[ru, dashrightarrow, "g_1"] & & z
    \end{tikzcd}%
\caption{Black and grey arrows in $\Gamma_x$, locally at a vertex $y$ in a minimal ultra summit set.}
\label{figproductoflechas}
\end{figure}

\begin{lemma}\label{productarrow} Let $USS(x)$ be minimal and let $y\in USS(x)$. Using the notations in~\autoref{figproductoflechas}, the following conditions hold:
\begin{itemize}
\item [\rm i)] $b_1 b_2$ is left-weighted.
\item [\rm ii)] $g_1 g_2$ is left-weighted.
\item [\rm iii)] $b_1 g_2=\Delta.$
\item [\rm iv)] $g_1 b_2=\Delta.$
\end{itemize}
\end{lemma}

\begin{proof}
By~\autoref{minimalrigidity}, we know that $y$ is rigid, and then $y^{-1}$ is also rigid. We recall that each black arrow is the initial factor of its source, and the final factor of its target. Also, each grey arrow is the initial factor of the inverse of its source, and the final factor of the inverse of its target. Hence:
\begin{itemize}
\item [i)] $b_1=\varphi(y)$ and $b_2=\iota(y)$. Since $y$ is rigid, $b_1 b_2$ is left-weighted.

\item [ii)] $g_1=\varphi(y^{-1})$ and $g_2=\iota(y^{-1})$. Since $y^{-1}$ is rigid, $g_1 g_2$ is left-weighted.

\item [iii)] $b_1=\varphi(y)$ and $g_2=\iota(y^{-1})=\partial(\varphi(y))$ by~\autoref{relacionfactoresxeinverso}. Hence $b_1 g_2=\Delta$.

\item [iv)] $b_2=\iota(y)$ and $g_1=\varphi(y^{-1})=\partial^{-1}(\iota(y))$ by~\autoref{relacionfactoresxeinverso}. Hence $g_1 b_2=\Delta$.

\end{itemize}
\end{proof}

\begin{corollary}\label{USSorbitas} If $USS(x)$ is minimal, then one of the following conditions holds:
\begin{itemize}
\item [\rm i)] $USS(x)$ has two orbits under cycling, conjugate to each other by $\Delta$
\item [\rm ii)] $USS(x)$ has one orbit under cycling, conjugate to itself by $\Delta$.
\end{itemize}
\end{corollary}

\begin{proof}
By~\autoref{minimalrigidity}, $USS(x)$ consists of rigid elements. As every rigid element belongs to $USS(x)$ (cycling is roughly a cyclic permutation of its factors), it follows that $USS(x)$ is the set of rigid conjugates of $x$. (This was already shown in~\cite[Theorem 1]{VolkerGMI}, anyway).

Let $y\in USS(x)$. As $USS(x)$ is minimal, every black arrow in $\Gamma_x$ corresponds to a cycling. Hence, iteratively conjugating $y$ by black arrows, we obtain the whole orbit of $y$ under cycling. Now consider the grey arrow $g_2$ starting at $y$, and denote by $b_1$ the black arrow ending at $y$ as in~\autoref{figproductoflechas}. Denote by $u$ the source of $b_1$ and by $z$ the target of $g_2$. By~\autoref{productarrow}  $b_1g_2=\Delta$, hence $u^{\Delta}=z$. Since $u$ belongs to the orbit of $y$ under cycling, say $\mathcal O_y$, it follows that $z$ belongs $\tau(\mathcal O_y)$.  This happens for all elements in $USS(x)$: Every grey arrow in $\Gamma_x$ sends an element in an orbit $\mathcal O$, to an element in $\tau(\mathcal O)$.

Therefore, if we start at an orbit $\mathcal O$, conjugating by a black arrow we stay in $\mathcal O$, and conjugating by a grey arrow we pass to the orbit $\tau(\mathcal O)$. As all elements in $USS(x)$ are connected by black and grey arrows, and $\Delta^2$ is a central element, it follows that in $USS(x)$ there are at most two orbits: If the orbit $\mathcal O$ is conjugate to itself by $\Delta$, there will be just one orbit; otherwise, there will be two orbits conjugate to each other by $\Delta$.
\end{proof}

Suppose that $USS(x)$ is minimal and has one orbit which is conjugate to itself by $\Delta$. For simplicity, we will assume that $x\in USS(x)$. We will distinguish two cases, depending on whether or not $\Delta$ commutes with $x$. If $\Delta\in Z(x)$ (the centralizer of $x$) then, as conjugation by $\Delta$ sends normal forms to normal forms, every factor in the left normal form of $x$ commutes with $\Delta$, hence every element in the orbit of $x$ is conjugate to itself by $\Delta$. That is, $\tau(y)=y$ for every element $y \in USS(x)$. If $\Delta\not\in Z(x)$, then $x$ is not conjugate to itself by $\Delta$ and this implies that the length of its orbit is an even number, as shown in the following result.

\begin{proposition}\label{longitudparorbita} Let $x\in USS(x)$ be a braid such that $USS(x)$ is minimal and has one orbit which is conjugate to itself by $\Delta$. Let $x=X_1,X_2,\ldots, X_k$ be the elements in this orbit, where $\textbf{c}(X_j)=X_{j+1}$ (indices are considered modulo $k$). If $\Delta\notin Z(x)$, then $k$ is even and $\tau(X_j)=X_{j+\frac{k}{2}}$ for $j=1,\ldots, k$.
\end{proposition}

\begin{proof}
Let $\Delta^p x_1\cdots x_l$ be the left normal form of $x$, and suppose that $p$ is even. Then $\tau(x)=\Delta^p\tau(x_1)\cdots\tau(x_l)\in USS(x)$ and $\tau(x)\neq x$ as $\Delta\not\in Z(x)$. Since $USS(x)$ has a single orbit, one can go from $x$ to $\tau(x)$ by iterated cycling, and since $x$ is rigid and $p$ is even, each cycling corresponds to a cyclic permutation of the factors in the left normal form of $x$. Hence, there is some $r$, $0<r<k$ such that $\tau(x)=\mathbf c^r(x)$. That is, $\tau(x_j)=x_{j+r}$ for all $j=1,\ldots, l$, where indices are taken modulo $l$. This implies that $x_j=\tau^2(x_j)=x_{j+2r}$, and hence $\mathbf c^{2r}(x)=x$. Since the orbit of $x$ under cycling has length $k$, it follows that $k|2r$. But since $k>r$, we finally obtain $k=2r$.

Therefore, the length of the orbit is an even number and $\tau(X_j)=X_{j+\frac{k}{2}}$ for every $X_j\in USS(x)$. Notice that $x=\Delta^p (x_1\cdots x_r)(\tau(x_1)\cdots\tau(x_r))\cdots(x_1\cdots x_r)(\tau(x_1)\cdots\tau(x_r))$.

Now suppose that $p$ is odd. In this case, since $x$ is rigid, the left normal form of $x^2$ is
$$
 x^2=\Delta^{2p} \tau(x_1)\ldots \tau(x_l)x_1\ldots x_l,
$$
and $x^2$ is also rigid. Hence, cycling $x^2$ corresponds to a cyclic permutation of the factors in its left normal form and, for every $i\geq 0$, the normal form of $\mathbf c^i(x^2)$ has $2l$ non-$\Delta$ factors, the last $l$ being the factors of $\mathbf c^i(x)$ and the first $l$ being their conjugates by $\Delta$. Therefore, the lengths of the orbits of $x$ and $x^2$ coincide, and the orbit of $x^2$ is precisely $x^2=(X_1)^2, (X_2)^2,\ldots,(X_k)^2$.

Since the infimum of $x^2$ is even, we can apply the previous case to $x^2$ and we obtain that $k$ is even, and that $\tau((X_j)^2)=(X_{j+\frac{k}{2}})^2$, which is equivalent to $\tau(X_j)=X_{j+\frac{k}{2}}$ for $j=1,\ldots,k$ (as the braids involved are rigid).
\end{proof}

Now that we know some properties of the structure of minimal ultra summit sets, we proceed to show that this case is generic.
Let us first show that the minimality of the ultra summit set is a local property:

\begin{theorem}\label{check_minimal_USS} Let $x$ be a rigid braid with $\ell(x)>1$. Then $USS(x)$ is minimal if and only if $\iota(x)$ and $\iota(x^{-1})$ are minimal simple elements for $x$.
\end{theorem}

\begin{proof}
The condition is clearly necessary. To show that it is sufficient, suppose that $\iota(x)$ and $\iota(x^{-1})$ are minimal simple elements for $x$. Let $\mathcal O_x$ be the orbit of $x$ under cycling, and let $\widetilde{\mathcal O}_x$ be the orbit of $x$ under twisted decycling.

It is shown in~\cite{BVGMI} that if $x$ is rigid and $\ell(x)>1$, then $USS(x)$ is the set of rigid conjugates of $x$. Then $USS(x^{-1})$ is the set of rigid conjugates of $x^{-1}$, so $USS(x^{-1})=USS(x)^{-1}$. Also, as all elements are rigid, decycling is just the inverse of cycling, hence twisted decycling sends any element in $\mathcal O_x$ to an element in $\tau(\mathcal O_x)$, and viceversa. It follows that $\mathcal O_x \cup \tau(\mathcal O_x) = \widetilde{\mathcal O}_x \cup \tau(\widetilde{\mathcal O}_x)$. Call this set $\mathcal U$.

By~\autoref{simplefactorNFsminimal}, as $\iota(x)$ is a minimal simple element for $x$, all elements in the normal form of $x$ are minimal, so $\iota(y)$ is a minimal simple element for $y$, for every $y\in \mathcal O_x$. Conjugating the whole picture by $\Delta$, it follows that $\iota(z)$ is a minimal simple element for $z$, for every $z$ in $\tau(\mathcal O_x)$. So we cannot escape from $\mathcal U=\mathcal O_x \cup \tau(\mathcal O_x)$ conjugating by black arrows, as all black arrows correspond to cyclings.

Now notice that one can obtain $\Gamma_{x^{-1}}$ from $\Gamma_x$ just by replacing each vertex $y$ with $y^{-1}$. The arrows will have the same labels, but their colors will be exchanged. Since $\iota(x^{-1})$ is a minimal simple element for $x$, it is also a minimal simple element for $x^{-1}$. Applying~\autoref{simplefactorNFsminimal} to $x^{-1}$, we have that all black arrows in $\mathcal O_{x^{-1}} \cup \tau(\mathcal O_{x^{-1}})$ correspond to cyclings. Notice that applying twisted decycling to a braid $y$ is equivalent to applying cycling to its inverse (the conjugating element is $\iota(y^{-1})$ in both cases). Hence $(\mathcal O_{x^{-1}})^{-1}= \widetilde{\mathcal O}_x$. Therefore, all grey arrows in $\widetilde{\mathcal O}_x \cup \tau(\widetilde{\mathcal O}_x)$ correspond to twisted decyclings. So we cannot escape from $\mathcal U=\widetilde{\mathcal O}_x \cup \tau(\widetilde{\mathcal O}_x)$ conjugating by grey arrows.

We have then shown that for every element $y\in \mathcal U$, $\iota(y)$ and $\iota(y^{-1})$ are minimal simple elements for $y$. Hence $y$ only admits these two minimal simple elements (as every minimal simple element must be a prefix of one of these, by~\autoref{caracterizacionelementosminimales}), and conjugating $y$ by any of these will take us again into $\mathcal U$. As we can obtain the whole $USS(x)$ starting with any element and conjugating by minimal simple elements, it follows that $\mathcal U=USS(x)$. Therefore, $USS(x)$ is minimal.
\end{proof}

\begin{theorem}\label{genericityUSSminimal} Let $x$ be a $\sigma_1$-non-intrusive braid. Then, $USS(x)$ is minimal.
\end{theorem}

\begin{proof}
By hypothesis, $x$ admits a non-intrusive conjugation to a rigid pseudo-Anosov braid $y$, where $\sigma_1$ and $\partial(\sigma_1)$ are simple factors of the normal form of $y$. $\sigma_1$ is clearly minimal as it cannot be decomposed, hence by \autoref{simplefactorNFsminimal} all simple factors in the normal form of $y$ are minimal. Therefore, $\iota(y)$ is a minimal simple element for $y$.

Now $\partial(\sigma_1)$ is also a simple factor of the normal form of $y$. This implies, using the relation between the normal forms of $y$ and $y^{-1}$, that some factor in the normal form of $y^{-1}$ equals either $\sigma_1$ or $\sigma_{n-1}$. Then, all simple factors in the normal form of $y^{-1}$ are minimal. Therefore $\iota(y^{-1})$ is a minimal simple element for $y^{-1}$, so it is a minimal simple element for $y$ (recall that $y$ is rigid).

Therefore, both $\iota(y)$ and $\iota(y^{-1})$ are minimal simple elements for $y$. Also, $\ell(y)>1$ as its left normal form contains both $\sigma_1$ and $\partial(\sigma_1)$. By~\autoref{check_minimal_USS} $USS(y)$, that is $USS(x)$, is minimal.
\end{proof}

The next result follows immediately from~\autoref{genericityUSSminimal} and~\autoref{genericitynonintrusivebraid}.

\begin{theorem}\label{resultadoprincipal}
  The proportion of braids in $\mathbf{B}(l)$ whose ultra summit set is minimal tends to 1 exponentially fast, as $l$ tends to infinity.
\end{theorem}

\section{The centralizer of generic braids}

We have now all the ingredients to describe the centralizer of a generic braid. We know from~\autoref{resultadoprincipal} that braids with minimal ultra summit sets are generic. Hence, we will consider braids $x$ such that $USS(x)$ is minimal, and we will describe an explicit set of generators for $Z(x)$, the centralizer of $x$. According to~\autoref{USSorbitas}, we can have two different situations, depending on whether $USS(x)$ has one or two orbits under cycling.

Recall that one can obtain an element $y\in USS(x)$ by applying iterated cyclic slidings (or iterated cyclings and decyclings) to $x$. This gives a conjugating element $c$ such that $y=c^{-1}x c$, and then $Z(y)= c^{-1} Z(x) c$. Therefore, in order to describe $Z(x)$ it suffices to describe $Z(y)$ where $y\in USS(x)$. We will then assume, for the rest of this section, that $x\in USS(x)$ and, in particular, that $x$ is rigid.

Following the ideas in~\cite{FrancoGMcentralizer}, we see that in order to compute a generating set for $Z(x)$, we just need to know $\Gamma_x$. Suppose that $a\in Z(x)$. Then $a=\Delta^{-2t}b$ for some $t\geq 0$, where $b$ is a positive braid. Then $b$ can be decomposed as a product of minimal simple elements, which corresponds to an oriented path in $\Gamma_x$ that starts and finishes at $x$. The positive braid $\Delta^2$ can also be decomposed as a loop in $\Gamma_x$ based at $x$. Hence, every element in $Z(x)$ can be decomposed as a product of loops in $\Gamma_x$, so a set of generators for $Z(x)$ is obtained from a set of generators of the fundamental group $\pi_1(\Gamma_x,x)$, replacing each loop by the braid it represents.

Computing a set of generators of $\pi_1(\Gamma_x,x)$ is a well-known procedure~\cite{LyndonSchupp}: Choose a maximal tree $T$ in $\Gamma_x$. For every vertex $v\in\Gamma_x$, call $\gamma_v$ the only simple path in $T$ going from $x$ to $v$. Let $A$ be the set of arrows in $\Gamma_x\setminus T$ and, for every $\lambda\in A$, denote $s$($\lambda$) and $t$($\lambda$) the source and the target of $\lambda$, respectively. Then $\pi_1$($\Gamma_x, x$) is generated by $F=\{\gamma_{s(\lambda)}\lambda\gamma^{-1}_{t(\lambda)} ; \;\lambda\in A\}$. If we denote by $\rho$ the homomorphism which maps $\pi_1$($\Gamma_x, x$) onto $Z(x)$, which sends each path to its associated element, then $\rho$($F$) is a generating set for $Z(x)$.

Let us then study the graphs $\Gamma_x$ with detail. We shall need the following.

\begin{definition}\rm Given $x\in B_n$, we define the \textit{preferred cycling conjugator} $PC(x)$ of $x$ as the product of conjugating elements corresponding to iterated cycling until the first repetition. That is, if $t$ is the smallest positive integer such that $\textbf{c}^{t}(x)=\textbf{c}^{i}(x)$ for some $0\leq i<t$, then:
$$PC(x)=\iota(x)\iota(\textbf{c}(x))\cdots\iota(\textbf{c}^{t-1}(x)).$$

Notice that if $x\in SSS(x)$ and one conjugates $x$ by $PC(x)$, one obtains an element in $USS(x)$. Notice also that if $x\in USS(x)$, then $PC(x)$ is the conjugating element along the whole cycling orbit of $x$. In particular, if $x\in USS(x)$ then $PC(x)$ commutes with $x$. Notice also that if $x$ is rigid, the above expression of $PC(x)$ is in left normal form as written. That is, the product $\iota(\textbf{c}^{i-1}(x))\iota(\textbf{c}^{i}(x))$ is left-weighted for every $i\geq 0$.
\end{definition}

Let us distinguish cases depending on the number of orbits in $USS(x)$.

\subsection{Ultra summit set with two orbits}

Let $x$ be a rigid braid such that $USS(x)$ is minimal and has two orbits under cycling, conjugated to each other by $\Delta$ (see~\autoref{USSorbitas}). Let $\mathcal{O}_1$ and $\mathcal{O}_2$ be these two orbits, and let $k$ be the length of each one. We can assume that $x\in\mathcal{O}_1$ and $\mathcal{O}_2=\tau(\mathcal{O}_1)$. Notice that $k=\ell(PC(x))$. We denote by $X_{1,j}$ and $X_{2,j}$ the elements in orbits $\mathcal{O}_1$ and $\mathcal{O}_2$, respectively, where $\tau(X_{i,j})=X_{3-i,j}$ and $\mathbf c (X_{i,j})=X_{i,j+1}$ for every $j=1,\ldots, k$, where $i=1,2$ and the second subindex is considered modulo $k$.

\begin{notation} Let $\Gamma_x$ be the graph associated to $USS(x)$. We denote $a_j$ the black arrow in $\mathcal{O}_1$ starting at $X_{1,j}$, $b_j$ the black arrow in $\mathcal{O}_2$ starting at $X_{2,j}$, $\alpha_j$ the grey arrow going from $\mathcal{O}_1$ to $\mathcal{O}_2$ ending at $X_{2,j}$ and $\beta_j$ the grey arrow going from $\mathcal{O}_2$ to $\mathcal{O}_1$ ending at $X_{1,j}$, for every $j=1,\ldots, k$. See~\autoref{figdosorbitas}, where grey arrows are represented by dashed lines.
\end{notation}

\begin{figure}[h!]
\centering
\includegraphics[width=7cm]{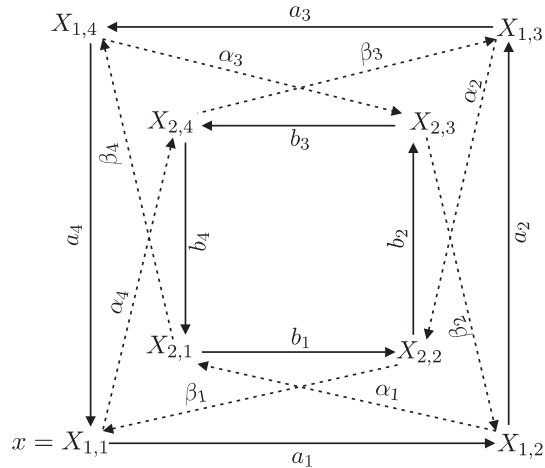}
\caption{$USS(x)$ with two orbits, for $k=4$.}
\label{figdosorbitas}
\end{figure}

\begin{theorem}\label{dosorbitas} Let $x$ be a rigid braid such that $USS(x)$ minimal and has two orbits under cycling. Then $Z(x)=\langle PC(x), \Delta^2\rangle$.
\begin{proof}
The inclusion $\langle PC(x), \Delta^2\rangle \subseteq Z(x)$ is clear, as $\Delta^2$ is central and $PC(x)\in Z(x)$ by construction. Let us show the converse inclusion by obtaining a set of generators for $Z(x)$, coming from a generating set for $\pi_1$($\Gamma_x, x$).

Let $T$ be the maximal subtree of $\Gamma_x$ whose arrows are $a_1,\ldots, a_{k-1}, b_1, \ldots, b_{k-1}, \alpha_1$. Every arrow $\lambda \in \Gamma_x\backslash T$ will produce a generator $F_\lambda = \gamma_{s(\lambda)} \lambda \gamma_{t(\lambda)}$ (by abuse of notation, we will identify the path in $\Gamma_x$ with the braid it represents).

We need to show that $F_{\lambda}$ can be written as a product of $PC(x)$, $\Delta^2$ and their inverses, for every $\lambda\in \Gamma_x\setminus T=\{a_k, b_k,\alpha_2,\ldots,\alpha_k,\beta_1,\ldots, \beta_k\}$.

First, $F_{a_k}=a_1\cdots a_k=PC(x)$, so the claim holds when $\lambda=a_k$.

Now $F_{b_k}=a_1 \alpha_1 b_1\cdots b_{k-1}b_k(a_1\alpha_1)^{-1}$. By~\autoref{productarrow}, $a_j\alpha_j=\Delta$ for every $j=1,\ldots, k$. Then $F_{b_k}=\Delta b_1\cdots b_k\Delta^{-1}$. Recall that the two orbits under cycling are conjugate by $\Delta$ and $\tau(X_{1,j})=X_{2,j}$ for every $j=1,\ldots, k$. This implies $\tau(a_j)=b_j$, hence $F_{b_k}=a_1\cdots a_k=PC(x)$.

Next, for $j=2,\ldots,k-1$ one has:
$$
F_{\alpha_j}=a_1\cdots a_j \alpha_j b_{j-1}^{-1}\cdots b_1^{-1}(a_1\alpha_1)^{-1} =a_1\cdots a_{j-1}\Delta b_{j-1}^{-1}\cdots b_1^{-1} \Delta^{-1}
$$
$$
= a_1\cdots a_{j-1} a^{-1}_{j-1}\cdots a^{-1}_1=1.
$$

Now
$$
F_{\alpha_k}=\alpha_k b_{k-1}^{-1}\cdots b_1^{-1}\alpha_1^{-1}a_1^{-1}=\alpha_k b_k b_k^{-1}b_{k-1}^{-1}\cdots b_1^{-1}\alpha_1^{-1}a_1^{-1}
$$
$$
=\Delta b_k^{-1}\cdots b_1^{-1}\Delta^{-1}=a_{k}^{-1}\cdots a_1^{-1}= PC(x)^{-1}.
$$

It remains to show that the loops determined by the $\beta_j$'s belong to  $\langle PC(x), \Delta^2\rangle$.

For $j=k$, $F_{\beta_k}=a_1\alpha_1\beta_k a_{k-1}^{-1}\cdots a_1^{-1}$. Notice that $b_k\beta_k=\Delta$ by~\autoref{productarrow}. Hence
$$
 F_{\beta_k}=a_1\alpha_1 b_k^{-1}\Delta a_{k-1}^{-1}\cdots a_1^{-1}=\Delta b_k^{-1}\Delta a_{k-1}^{-1}\cdots a_1^{-1}=\Delta^2 PC(x)^{-1}.
$$

Finally, for $j=1,\ldots,k-1$, we have:
$$
F_{\beta_j}= a_1\alpha_1 b_1\cdots b_j \beta_j a_{j-1}^{-1}\cdots a_1^{-1}= \Delta b_1\cdots b_{j-1} \Delta a_{j-1}^{-1}\cdots a_1^{-1} = \Delta^2.
$$
\end{proof}
\end{theorem}

\subsection{Ultra summit set with one orbit}

Consider now a rigid braid $x$ such that $USS(x)$ is minimal and has only one orbit $\mathcal{O}_1$ of length $k$, conjugated to itself by $\Delta$ (see~\autoref{USSorbitas}). Recall that $k=\ell(PC(x))$. We denote by $X_{1,j}$ the elements in this orbit, for every $j=1,\ldots, k$, where $x=X_{1,1}$. We will distinguish two different cases, depending on whether $\Delta$ belongs to $Z(x)$ or not. By~\autoref{longitudparorbita}, if $\Delta$ does not belong to $Z(x)$, then $k$ is even and $\tau(X_{1,j})=X_{1,j+\frac{k}{2}}$ (where $j+\frac{k}{2}$ is considered modulo $k$) for every $j=1, \ldots, k$. Otherwise, $\tau(X_{1,j})=X_{1,j}$, for all $j=1\ldots, k$. \autoref{figunaorbita} illustrates the two cases.

\begin{notation} \rm Let $\Gamma_x$ be the graph which represents the whole set $USS(x)$ in any of the above two cases. We denote $a_j$ the black arrow starting at $X_{1,j}$, and $\alpha_j$ the grey arrow starting at $X_{1,j+1}$, so that $a_j\alpha_j=\Delta$, for every $j=1,\ldots, k$.
\end{notation}

\begin{figure}[H]
\centering
\subfigure{\includegraphics[width=70mm]{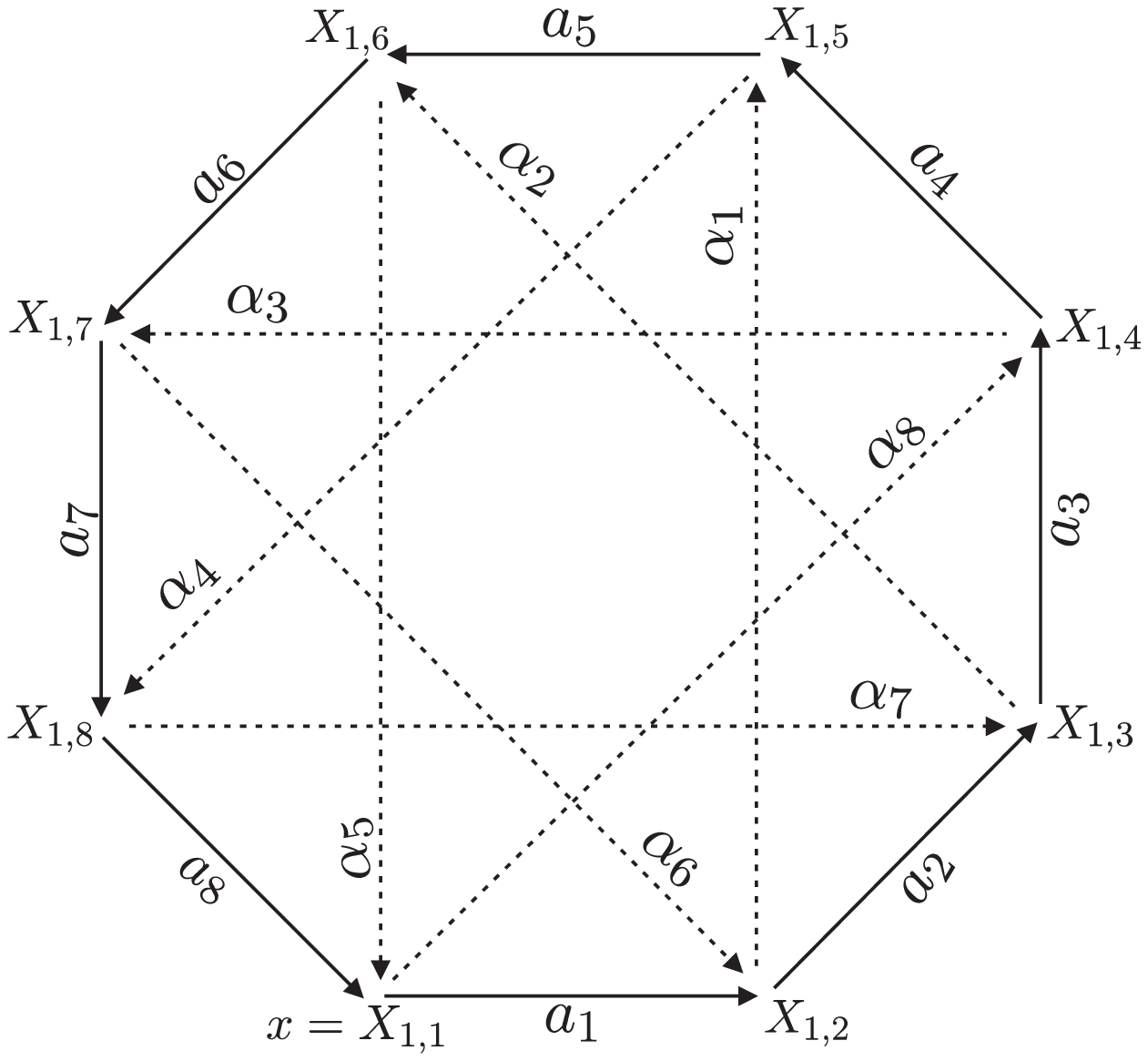}}\hspace{5mm}
\subfigure{\includegraphics[width=70mm]{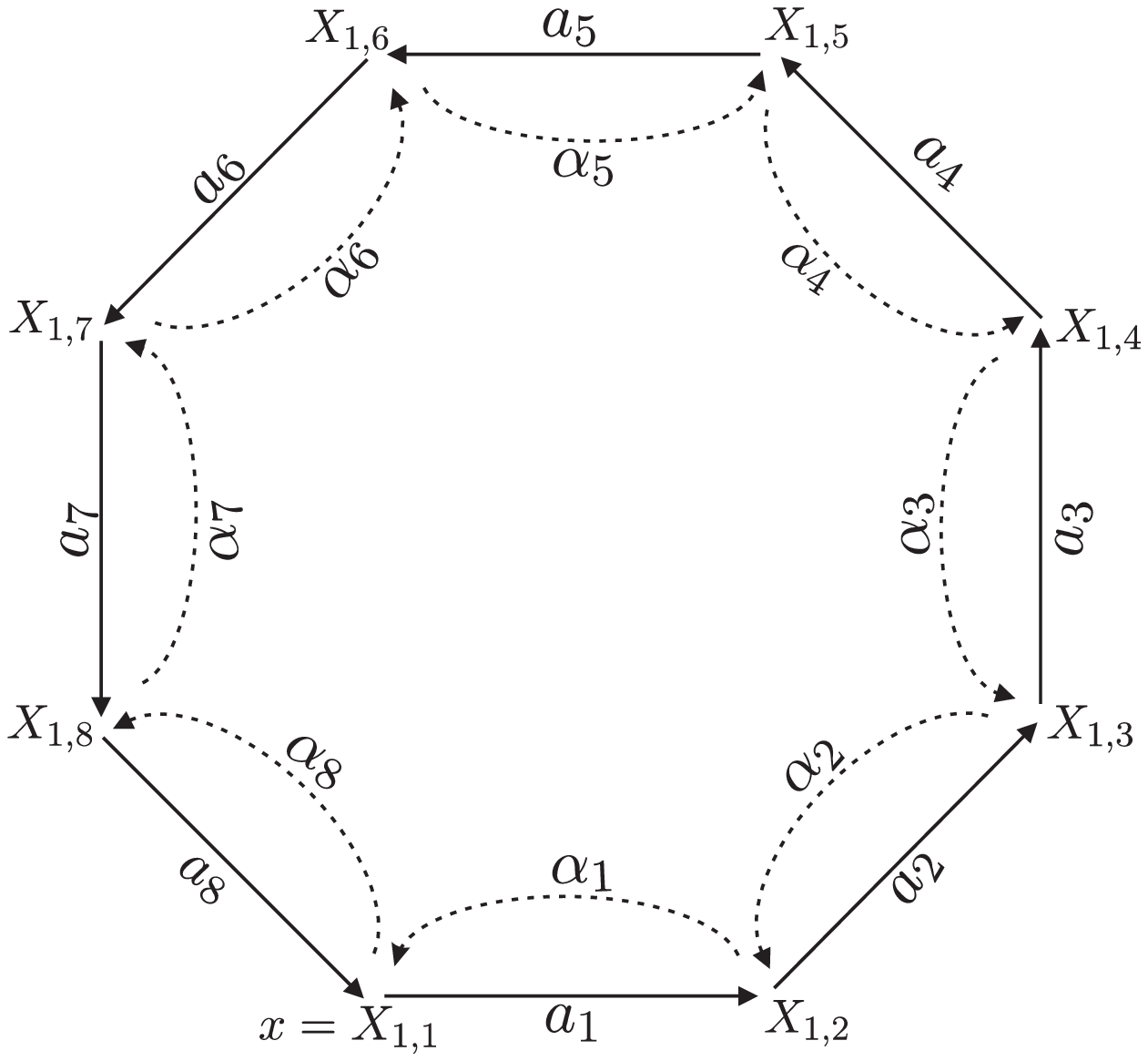}}
\caption{$USS(x)$ with one orbit such that $\Delta^{-1}x\Delta\neq x$ (left) and $\Delta^{-1}x\Delta=x$ (right), for $k=8$.} \label{figunaorbita}
\end{figure}

\begin{theorem}\label{unaorbita} Let $x$ be a rigid braid such that USS(x) is minimal and has only one orbit under cycling of length $k$. One has:
\begin{itemize}
\item [\rm i)] If $\Delta^{-1}x\Delta\neq x$, then $Z(x)=\left\langle a_1\cdots a_{\frac{k}{2}} \Delta^{-1}, \; \Delta^2\right\rangle$.
\item [\rm ii)] If $\Delta^{-1}x\Delta=x$, then $Z(x)=\left\langle PC(x), \Delta\right\rangle$.
\end{itemize}
\end{theorem}

\begin{proof}
In both cases, the inclusions $\left\langle a_1\ldots a_{\frac{k}{2}} \Delta^{-1},\; \Delta^2\right\rangle\subseteq Z(x)$ and $\langle PC(x), \Delta\rangle\subseteq Z(x)$ are immediate.

We obtain a generating set for $Z(x)$ in a similar way as it is done in the proof of \autoref{dosorbitas}. Let $T$ be the maximal tree in $\Gamma_x$, with base vertex $x=X_{1,1}$, containing the edges $a_1\ldots, a_{k-1}$.

Suppose first that $\Delta^{-1}x\Delta\neq x$. We need to show that $F_{\alpha_1},\ldots,F_{\alpha_k},F_{a_k} \in \left\langle a_1\ldots a_{\frac{k}{2}} \Delta^{-1},\; \Delta^2\right\rangle$.

Denote $g=a_1\cdots a_{\frac{k}{2}} \Delta^{-1}$. We claim that
$$
a_1\cdots a_j\left(\alpha_j a_{j+\frac{k}{2}-1}^{-1}a_{j+\frac{k}{2}-2}^{-1}\cdots a_{j+1}^{-1}\right)a_{j}^{-1}\cdots a_1^{-1}=g^{-1}
$$
for every $j=1,\ldots,k$, where every subindex is considered modulo $k$. In order to prove this claim, applying~\autoref{productarrow} and~\autoref{longitudparorbita}, we recall that $a_j\alpha_j=\Delta$ and $\tau(X_{1,j})=X_{1,j+\frac{k}{2}}$. Hence $\tau(a_j)=a_{j+\frac{k}{2}}$ and then the left hand side of the formula is equal to
$$
 a_1\cdots a_{j-1}\Delta a_{j+\frac{k}{2}-1}^{-1} \cdots a_1^{-1}
 =\Delta a_{1+\frac{k}{2}} \cdots a_{j-1+\frac{k}{2}} a_{j+\frac{k}{2}-1}^{-1} \cdots a_1^{-1}
 =\Delta a_{\frac{k}{2}}^{-1} \cdots a_1^{-1}= g^{-1},
$$
proving the claim.

It follows immediately that $F_{\alpha_{j}}=g^{-1}$, for $j=1, \ldots, \frac{k}{2}$.

For $j=\frac{k}{2}+1, \ldots, k-1$, one has:
$$
  F_{\alpha_j}= a_1\cdots a_j \alpha_j a_{j-\frac{k}{2}-1}^{-1}\cdots a_1^{-1} = a_1\cdots a_{j-1} \Delta a_{j-\frac{k}{2}-1}^{-1}\cdots a_1^{-1}
$$
$$
  = a_1\cdots a_{j-1} a_{j-1}^{-1} \cdots a_{\frac{k}{2}+1}^{-1} \Delta = a_1 \cdots a_{\frac{k}{2}} \Delta = g \Delta^2.
$$

For $j=k$:
$$
F_{\alpha_{k}}=\alpha_{k}a_{\frac{k}{2}-1}^{-1}\cdots a_1^{-1}=a_k^{-1}(a_k\alpha_k) a_{\frac{k}{2}-1}^{-1}\cdots a_1^{-1} =  \Delta a_{\frac{k}{2}}^{-1}\cdots a_1^{-1}=g^{-1}.
$$
Finally:
$$
 F_{a_{k}}=a_1\cdots a_{k}=a_1\cdots a_{\frac{k}{2}}a_{\frac{k}{2}+1}\cdots a_{k} = a_1\cdots a_{\frac{k}{2}}(\Delta^{-1}a_1\cdots a_{\frac{k}{2}-1}\Delta) = g^2 \Delta^2.
$$
This shows the result when $\Delta^{-1}x\Delta\neq x$.

Suppose now that $\Delta^{-1}x\Delta=x$. We have to show that $F_{\alpha_1},\ldots,F_{\alpha_k}, F_{a_k} \in \left\langle PC(x), \Delta \right\rangle$.

For $j=1,\ldots, k-1$ we have $F_{\alpha_j}=a_1\cdots a_{j-1}(a_j\alpha_j) a_{j-1}^{-1}\cdots a_1^{-1}$. Recall that $a_j\alpha_j=\Delta$ and that $\tau(X_{1,j})=X_{1,j}$, so $\tau(a_i)=a_i$ for every $i$. Hence $F_{\alpha_j}=\Delta$ for $j=1,\ldots, k-1$.

Now $F_{\alpha_{k}}=\alpha_{k}a_{k-1}^{-1}\cdots a_1^{-1}= (\alpha_k a_k) a_k^{-1}\cdots a_1^{-1}=\Delta(PC(x))^{-1}$.

Finally, $F_{a_{k}}=a_1\cdots a_{k}= PC(x)$.
\end{proof}

\section{An algorithm to compute the centralizer of a braid}

Our goal in this section is to present an algorithm to compute a generating set for the centralizer of a braid, based in the results from previous sections, which generically terminates in polynomial time and yields a minimal set of generators consisting of two elements.

Given $y\in B_n$, the first step will be to find $x\in USS(y)$ and a conjugating element $c$ such that $c^{-1}yc=x$. This is achieved by iterated applications of cyclic sliding to $y$ \cite{VolkerGMI}. Next we check whether $x$ is rigid. If $x$ is rigid, the algorithm continues by computing the minimal simple elements for $x$ \cite{FrancoGMconjugacy}. If $\iota(x)$ and $\iota(x^{-1})$ are equal to the minimal simple elements for $x$, then $USS(y)=USS(x)$ is minimal by~\autoref{check_minimal_USS}. At this point, one needs to determine whether $USS(y)$ has one or two orbits, and in the latter case whether $x$ commutes with $\Delta$, in order to apply~\autoref{dosorbitas} or~\autoref{unaorbita}. For that, we will check whether or not $\tau(x)$ belongs to the orbit of $x$, and in particular whether $\tau(x)$ is equal to $x$. This yields a generating set for $Z(x)$ given by~\autoref{dosorbitas} or~\autoref{unaorbita} and one obtains a generating set for $Z(y)$ from the equality $Z(y)=cZ(x)c^{-1}$.

We remark that if $x$ is not rigid or if $USS(y)$ is not minimal (a case which is not generic, as shown in~\autoref{resultadoprincipal}), the centralizer of $y$ can be determined by the algorithm given by Franco and Gonz\'alez-Meneses~\cite{FrancoGMcentralizer}. The algorithm in~\cite{FrancoGMcentralizer} refers to $SSS(x)$ instead of $USS(x)$, but it can be equally applied to $USS(x)$, or even to $SC(x)$, considerably improving its efficiency.

The pseudocode of the algorithm described in the previous paragraphs is shown in~\autoref{Centralizer:Alg}.

\begin{algorithm}[ht]
  \caption{Compute a generating set for the centralizer of a braid.}
  \label{Centralizer:Alg}
  \begin{algorithmic}[1]
    \REQUIRE A braid $y\in B_n$ given in left normal form.
    \ENSURE A generating set for $Z(y)$
    \smallskip

    \STATE Compute $x\in USS(y)$ and $c\in B_n$ such that $c^{-1}yc=x$ by iterated cyclic sliding~\cite{VolkerGMI, VolkerGMII}.
    \IF{$x$ is rigid}
        \STATE Applying iterated cycling to $x$, compute $k=$ length of the orbit of $x$ under cycling.
        \STATE Compute the minimal simple elements for $x$~\cite{Volker}.
        \IF{$\iota(x)$ and $\iota(x^{-1})$ are the minimal simple elements for $x$}
          \STATE Compute $\tau(x)$.
          \IF{$\tau(x)=x$}
           \STATE $A=PC(x)$, $B=\Delta$. [Theorem~\ref{unaorbita}]
           \ELSIF{$k$ is even \AND $\tau(x)=\textbf{c}^{\frac{k}{2}}(x)$}
             \STATE $A= a_1\cdots a_{\frac{k}{2}}\Delta^{-1}$, $B=\Delta^2$. [Theorem~\ref{unaorbita}]
             \ELSE
               \STATE $A=PC(x)$, $B=\Delta^2$. [Theorem~\ref{dosorbitas}]
          \ENDIF
              \RETURN \{$cAc^{-1}$, $cBc^{-1}$\}.
          \ELSE
           \STATE Compute $Z(x)$ applying the algorithm in \cite{FrancoGMcentralizer} (using $SC(x)$).
        \ENDIF
        \ELSE
        \STATE Compute $Z(x)$ applying the algorithm in \cite{FrancoGMcentralizer} (using $SC(x)$).
    \ENDIF

  \end{algorithmic}
\end{algorithm}

In order to study the complexity of this algorithm, we naturally measure the input braids by their canonical length. But we will also see how the number $n$ of strands affects the computation, as this is important in practice, and because the algorithms in braid groups are usually programmed allowing inputs with an arbitrary number of strands.

We saw in~\autoref{resultadoprincipal} that $\sigma_1$-non-intrusive braids are generic in $B_n$, so we will study the complexity of computing the centralizer of such a braid.

\begin{proposition} Let $y\in B_n$ be a $\sigma_1$-non-intrusive braid such that $USS(y)$ is minimal. If $l=\ell(y)$, the complexity of computing a generating set for $Z(y)$ using~\autoref{Centralizer:Alg} is $\textit{O}(l^2n^4\log n)$.
\end{proposition}

\begin{proof}
The first step of~\autoref{Centralizer:Alg} consists of applying iterated cyclic sliding until the first repeated element $x$ is obtained, computing at the same time the conjugating element $c$, which is the product of the conjugating elements for each cyclic sliding. By~\cite[Algorithm~1]{VolkerGMII}, the cost of computing $x$ and $c$ is \textit{O}($Tln\log n$), where $T$ is the number of cyclic slidings applied.

Recall from~\autoref{minimalrigidity} that, as $USS(y)$ is minimal, $x$ is rigid. In~\cite{VolkerGMI} it is shown that, if a braid $y$ is conjugate to a rigid braid $x$, then the product of all conjugating elements corresponding to iterated cyclic sliding is the shortest positive element conjugating $y$ to a rigid braid. On the other hand, as $y=\Delta^p y_1\cdots y_l$ is $\sigma_1$-non-intrusive, there is a positive prefix of $\tau^{-p}(y_1\cdots y_l)$ conjugating $y$ to a rigid element. This implies that the number of cyclic slidings needed to conjugate $y$ to a rigid braid is smaller than the letter length of $\tau^{-p}(y_1\cdots y_l)$, that is, smaller than $l\|\Delta\|=ln(n-1)/2$. Since the obtained element $x$ is rigid, the next cyclic sliding yields $\mathfrak s(x)=x$ and the repeated element is immediately obtained. Hence $T=O(ln^2)$, and the first step of the algorithm takes time $O(l^2 n^3\log n)$.

The following step is to check whether $x$ is rigid, which means to determine whether $\mathfrak p(x)=1$. This was already made in the previous step, when $\mathfrak s(x)$ was computed. So this step has no cost.

Let $x=\Delta^q x_1\cdots x_r$ in left normal form. Since $x\in USS(y)\subset SSS(y)$, we have $q\geq p$ and $r\leq l$. Since $x$ is rigid, the next step in~\autoref{Centralizer:Alg} is to apply iterated cycling to $x$ in order to compute $k$, the length of its orbit under cycling. As $x$ is rigid, the left normal form of $\mathbf c(x)$ is precisely $\Delta^{p} x_2\cdots x_r \tau^{-p}(x_1)$. In order to compute this normal form, one just needs to apply $\tau^{-p}$ to $x_1$. If $p$ is even then $\tau^{-p}$ is trivial, and if $p$ is odd $\tau^{-p}$ is equal to $\tau$. If the simple factors are stored, as usual, as a list of $n$ numbers (from 1 to $n$) representing the permutation they induce, applying $\tau$ to $x_1$ is $O(n)$. Hence, each cycling takes $O(n)$. After each cycling, one need to check whether the obtained element equals $x$. Comparing two normal forms is $O(ln\log n)$ (we compare two lists of $rn\leq ln$ numbers between 1 and $n$, and each number has length at most $\log n$). We know that $\mathbf c^{2r}(x)=x$, hence $k\leq 2r \leq 2l$, so the number of total comparisons we will perform is $O(l)$. Therefore, in order to compute $k$ one performs $k\leq 2l$ cyclings (total cost $O(ln)$) and $O(l)$ comparisons (total cost $O(l^2n\log n)$). Hence, the cost of computing $k$ is $O(l^2n\log n)$.

The next step consists of computing the minimal simple elements for $x$. The complexity of this computation is described in~\cite[Proposition 4.10]{BVGMII} for general elements in a Garside group, but in our case we are in a simpler situation: We are computing the minimal simple elements for a braid $x\in B_n$ which is rigid. We can proceed as follows: For every generator $\sigma_i$ ($i=1,\ldots,n-1$), compute the minimal positive element $\rho_i$ such that $\sigma_i\preccurlyeq \rho_i$ and $x^{\rho_i}\in SSS(x)$. By~\cite[Proposition 3.4]{VolkerGMII}, in order to compute this element we start with $\rho=\sigma_i$, and while $\ell(x^\rho)>r$ we replace $\rho$ with $\rho\left(1\vee (x^\rho)^{-1}\Delta^p \vee x^\rho \Delta^{-p-r}\right)$. The final value of $\rho$ is precisely $\rho_i$. The element $1\vee (x^\rho)^{-1}\Delta^p$ is the leftmost factor in the right normal form of $(x^\rho)^{-1}\Delta^p$, and $1\vee x^\rho \Delta^{-p-r}$ is the leftmost factor in the right normal form of $x^\rho \Delta^{-p-r}$. Hence, the element $1\vee (x^\rho)^{-1}\Delta^p \vee x^\rho \Delta^{-p-r}$ can be obtained at the cost of computing two right normal forms and the least common multiple of two simple elements. The total cost of one iteration is $O(l^2n\log n)$.

Once $\rho_i$ is computed, we can use~\cite[Theorem 2]{VolkerGMI}: Since $x^{\rho_i}\in SSS(y)$ and it is conjugate to a rigid element, the shortest positive element $\alpha_i$ conjugating $x^{\rho_i}$ to a rigid element is precisely the product of all conjugating elements for iterated cyclic sliding of $x^{\rho_i}$. By construction $\rho_i\alpha_i$ will be the minimal element which admits $\sigma_i$ as a prefix and conjugates $x$ to a rigid element. Since $\Delta$ admits $\sigma_i$ as a prefix and conjugates $x$ to a rigid element, it follows that $\rho_i$ is a prefix of $\Delta$, so it is simple. Hence, the total number of iterations of the procedure in this paragraph and the procedure in the previous paragraph is bounded by $n(n-1)/2$. Since the cost of a single cyclic sliding is $O(ln\log n)$, it follows that the complexity of computing $\rho_i \alpha_i$ is $O(l^2n^3\log n)$. This is done for $i=1,\ldots, n-1$, and the set of minimal simple elements for $x$ is just the set of minimal elements in $\{\rho_1\alpha_1,\cdots,\rho_{n-1}\alpha_{n-1}\}$, hence we can compute this set in $O(l^2 n^4\log n)$.

The next steps, computing $\tau(x)$ and comparing it with $x$ or $\mathbf c^{\frac{k}{2}}(x)$ (which has been computed in a previous step, when $k$ was obtained), are negligible compared with the previous ones.

Finally, it remains to conjugate the generators $A$ and $B$ by $c^{-1}$ to obtain $Z(y)$. Since $y$ is $\sigma_1$-non-intrusive, $\ell(c)\leq \frac{2}{5}\ell(x)=O(l)$, and on the other hand $\ell(A)$ is at most $k$ in all possible cases, where $k\leq 2l$. Then, the cost of conjugating $A$ and $B$ by $c^{-1}$, is \textit{O}($l^2n\log n$)~\cite{VolkerGMII}.

Therefore, the total time complexity of the whole algorithm, when applied to the $\sigma_1$-non-intrusive braid $y$, is $O(l^2n^4\log n)$.
\end{proof}

\begin{corollary}\label{C:polynomial_algorithm} There exists an algorithm to compute a generating set for the centralizer of a braid $y\in B_n$, whose generic-case complexity is $\textit{O}(l^2n^4\log n)$, where $l=\ell(y)$.
\end{corollary}

Notice that, when $B_n$ is fixed, the generic-case complexity of~\autoref{Centralizer:Alg} is quadratic on the canonical length of the input braid. But one can implement the algorithm letting $n$ vary, and the generic-case complexity still remains polynomial.

\medskip

{\bf Juan Gonz\'alez-Meneses} (meneses@us.es), \;  {\bf Dolores Valladares} (valladaresg@us.es). \\
Depto. de \'Algebra. Instituto de Matem\'aticas (IMUS). \\
Universidad de Sevilla.  Av. Reina Mercedes s/n, 41012 Sevilla (Spain).

\end{document}